\documentclass[11pt]{amsart}
\usepackage{amsmath,amssymb,amsthm}
\usepackage{mathrsfs}
\usepackage{graphicx}
\usepackage{enumerate}

\theoremstyle{plain}
\newtheorem*{theorem}{Theorem}
\newtheorem*{corollary}{Corollary}

\newcommand{\id}{\mathrm{id}}

\makeatletter
\def\ps@pprintTitle{%
  \let\@oddhead\@empty
  \let\@evenhead\@empty
  \def\@oddfoot{\reset@font\hfil\thepage\hfil}
  \let\@evenfoot\@oddfoot
}
\makeatother

\begin{document}

\title{Separating domains from algebraic domains}

\author{Xiaodong Jia, Qingguo Li, Wei luan}
\address{X.\,Jia,  School of Mathematics, Hunan University, Changsha, Hunan, 410082, China. Email:  {\rm jiaxiaodong@hnu.edu.cn} }
\address{Q.\,Li, School of Mathematics, Hunan University, Changsha, Hunan, 410082, China.  Email: {\rm liqingguoli@aliyun.com}}
\address{W.\,Luan (corresponding author), School of Mathematics, Hunan Normal University, Changsha, Hunan, 410082, China.  Email: {\rm luanwei@hunnu.edu.cn}}

\begin{abstract}
We prove that every domain that fails to be algebraic admits the unit interval $[0, 1]$ as its Scott-continuous retract. As a result, every countable domain is algebraic. 
\end{abstract}
\maketitle

Domain theory \cite{gierz03, abramsky94, goubault13a}, initially invented by Dana Scott, is a theory that studies the structure of different semantic domains and is applied, in particular, to denotational semantics to model programming languages \cite{streicher06}. 

Domains, in a nutshell,  are {\bf d}irected-{\bf c}omplete {\bf po}sets (dcpo's, for short) on which the \emph{way-below relation} is \emph{approximating}. Concretely, a dcpo $P$ is a partially ordered set $(P, \leq)$ in which every directed subset has a supremum. The way-below relation (in symbol $\ll$) on $P$ is defined as $x\ll y$ if for every directed subset $D$ of $P$, that $y\leq \sup D$ implies that $x\leq d$ for some $d\in D$. The relation $\ll$ on $P$ is approximating if the set $\{x\in P\mid x\ll p\}$ is a directed set with $p$ as its supremum for all $p\in P$. That is, $P$ is a domain if and only if $\{x\in P \mid x\ll p\}$ is directed and $p = \sup \{x\in P \mid x\ll p\}$ for all $p\in P$. Finite posets are  examples of domains. The unit interval $[0, 1]$ is a domain with the usual ordering, and one could see that $x\ll y$ if and only if $x< y$ or $x = 0$. 
\emph{Algebraic} domains are domains where we put more conditions on the approximating property: a domain $P$ is called algebraic if for all $p \in P$, the subset 
$\{x\in P \mid x\ll x ~\text{and}~ x\ll p  \}$ is directed and $p = \sup \{x\in P \mid x\ll x ~\text{and}~ x\ll p  \} $ for all $p\in P$. Elements $x$ with the property $x\ll x$ are called \emph{compact elements}, and hence algebraicity of domains reads as every element can be approximated by compact elements.  Every finite poset is algebraic as every element in it is compact, and this applies to posets of finite heights. The unit interval $[0, 1]$ serves as an example that is a  domain but not algebraic, as $0$ is the only compact element in it.
Domains and algebraic domains are closely related, in the sense that each domain is actually a \emph{Scott-continuous retract} of an algebraic domain \cite[Theorem I-4.17]{gierz03}. In general, a map $f\colon P\to Q$ between dcpo's $P$ and $Q$ is Scott-continuous if it preserves directed suprema. That is, $f$ is monotone and $f(\sup D) = \sup f(D)$ for all directed subsets $D$ of $P$; and the dcpo $Q$ is called a Scott-continuous retract of $P$ provided that there is a Scott-continuous map $g\colon Q\to P$ with $f\circ g = \id_Q$, and in this case, $f$ is called a \emph{retraction} and $g$ a \emph{section}.

The smaller collection of algebraic domains are rich enough to denote data types in early programming languages. It was the purpose to model probabilistic behaviours in programming that domains are introduced to encompass the domain of (extended) reals or simply the unit interval $[0, 1]$ to talk about probabilities. We have seen that $[0, 1]$ is never algebraic, so it separates domains from algebraic domains. Somewhat to one's surprise, we report in this note that the unit interval $[0, 1]$ is \emph{essentially the unique} example that separates domains from algebraic domains: 

\begin{theorem} \label{main}
The unit interval $[0, 1]$ is a Scott-continuous retract of every domain that fails to be algebraic.
\end{theorem}
\begin{proof}
Assume that $P$ is a domain and $P$ is not algebraic. Then there exist $p, q \in P$ such that $p\ll q$ and there are no compact elements between them, in particular, neither $p$ nor $q$ is compact; otherwise, every pair of such $p, q$ can be interpolated by compact elements, and then by \cite[Proposition~I-4.3]{gierz03} $P$ would have been algebraic. 

Now by the Interpolation Property of the way-below relation on~$P$ \cite[Theorem~I-1.9 (ii)]{gierz03}, we could find $a_\frac{1}{2}\in P$ such that $p\ll a_{\frac{1}{2}} \ll q$. Again by the Interpolation Property, we find $a_{\frac{1}{4}}$ and $a_{\frac{3}{4}}$ such that $p\ll a_{\frac{1}{4}}\ll a_{\frac{1}{2}} \ll a_{\frac{3}{4}}\ll q$. Repeat this process and we could get a sequence of elements indexed by dyadic numbers in  $(0, 1)$, and $a_{\frac{l}{2^n}} \ll a_{\frac{k}{2^m}}$ if and only if $\frac{l}{2^n} < \frac{k}{2^m}$. Moreover, none of them is compact as they are between $p$ and~$q$, hence they are all distinct from each other. 

Let $\mathbb D$ be the set of all dyadic numbers in $(0, 1)$. We consider the map $g\colon [0, 1] \to P$:
$$ g(r) = \begin{cases} 
p, & r= 0;\\
\sup \{a_d \mid d \in \mathbb D~\text{and}~d< r \}, & otherwise. 
\end{cases}$$
This map is well-defined since the set $\{a_d \mid d \in \mathbb D~\text{and}~d< r \}$ is directed for $r>0$. Obviously, $g$ is monotone. Moreover, $g$ is Scott-continuous. To see this, we assume $(r_i)_{i\in I}$ is a family of real numbers in $[0,1]$ and $r = \sup_{i\in I} r_i$. As every dyadic number $d$ strictly below $r$ must be strictly below some $r_i$, so we have $a_d\leq g(r_i)$ and this proves that $g(r) \leq \sup_{i\in I} g(r_i)$. That $g(r) \geq \sup_{i\in I} g(r_i)$ holds since $g$ is monotone. So we have proved that $g$ is Scott-continuous. 

We proceed to define a map $f\colon P \to [0, 1]$:
$$ f(x) =  \sup \{d\in \mathbb D \mid a_d \ll x \}. $$
The map $f$ is well-defined since $[0, 1]$ is a complete lattice. It is easy to see that $f$ is monotone. 
For Scott-continuity of $f$, we assume that $(x_i)_{i\in I}$ is a directed family in $P$ and $x = \sup_{i\in I}x_i$. For $a_d \ll x$, by the interpolation property of the way-below relation on~$P$ we could find $y$ with  $a_d \ll y \ll x$. So we know that $a_d\ll y \leq x_i$ for some~$x_i$. This implies $d \leq f(x_i)$ and hence $f(x) \leq \sup_{i\in I} f(x_i)$. Again, that $f(x) \geq \sup_{i\in I} f(x_i)$ holds since $f$ is monotone. So $f$ is indeed Scott-continuous. 

Finally, we prove that $f\circ g = \id_{[0, 1]}$. Obviously $f(g(0)) = f(p) = \sup \emptyset = 0$. Now take $r \in (0, 1]$. By definition $ f(g(r)) =  \sup \{d\in \mathbb D \mid a_d \ll g(r) \}$. If $a_e \ll g(r) = \sup \{a_d \mid d \in \mathbb D~\text{and}~d< r \}$ for some $e\in \mathbb D$,  then $a_e \ll a_d$ for some $d\in \mathbb D$ and $d<r$. So by definition $e< d < r$, and hence $f(g(r))$, as the supremum of all such $e$'s, is below $r$, i.e. $f(g(r)) \leq r$. Conversely, we notice that for each $r\in (0, 1]$ and each dyadic number $e< r$, there exists another dyadic number $d$ with $e< d < r$. So $a_e \ll a_d \leq g(r)$; therefore, $e \in \{d\in \mathbb D \mid a_d \ll g(r)\}$. Hence $f(g(r))$ should be above all such $e$'s and above their supremum $r$.
\end{proof}

For the cardinality issue, $[0, 1]$ cannot be retracts of any countable domains, hence we know the following result, which seems to the authors new to the community.
\begin{corollary}
Countable domains are algebraic.  \hfill $\Box$
\end{corollary}

\section*{Acknowledgement}
The first author thanks Rongqi Xiao for asking related questions and Jean Goubault-Larrecq for useful discussions. The authors acknowledge support by NSFC (No.\,12371457, No.\,12231007 and No.\,12301583).

\end{document}